\documentclass{amsart}
\usepackage{amsthm,enumerate}
\usepackage{amsmath,amssymb}
\newtheorem{theorem}{Theorem}
\newtheorem{cor}[theorem]{Corollary}

\newtheorem{lemma}[theorem]{Lemma}
\newtheorem{definition}[theorem]{Definition}

\numberwithin{theorem}{section}

\allowdisplaybreaks
\newcommand{\Om} {\Omega}
\newcommand{\pa} {\partial}
\newcommand{\be} {\begin{equation}}
\newcommand{\ee} {\end{equation}}
\newcommand{\bea} {\begin{eqnarray}}
\newcommand{\eea} {\end{eqnarray}}
\newcommand{\Bea} {\begin{eqnarray*}}
\newcommand{\Eea} {\end{eqnarray*}}

\newcommand{\al} {\alpha}

\newcommand{\De} {\Delta}
\newcommand{\la} {\lambda}

\newcommand{\h}{\mathbb{H}^n}
\newcommand{\Dh}{\Delta_{\mathbb{H}^n}}
\newcommand{\nh}{\nabla_{\mathbb{H}^n}}
\newcommand{\norm}[1]{\left\Vert#1\right\Vert}
\def\R{{\mathbb R}}
\def\H{{\mathbb H}}
\newcommand{\no} {\nonumber}
\newcommand{\lab} {\label}

\newcommand{\rar}{\rightarrow}
\numberwithin{equation}{section}

\begin{document}
\title[Singular Adams  inequality for biharmonic operator ]{Singular Adams inequality for biharmonic operator on Heisenberg Group and its applications}
\author[G.Dwivedi,\, J.\,Tyagi ]
{ G. Dwivedi,  J.Tyagi }

\address{G.\,Dwivedi \hfill\break
 Indian Institute of Technology Gandhinagar \newline
 Palaj, Gandhinagar \newline
 Gujarat, India - 382355}
 \email{dwivedi\_gaurav@iitgn.ac.in}

\address{J.\,Tyagi \hfill\break
 Indian Institute of Technology Gandhinagar \newline
 Palaj, Gandhinagar \newline
 Gujarat, India - 382355}
 \email{jtyagi@iitgn.ac.in, jtyagi1@gmail.com}

\begin{abstract}The goal of this paper is to establish singular Adams type inequality for biharmonic operator on Heisenberg group. As an application, we establish
the existence of a solution to
\begin{equation*}
\Delta_{\H}^2 u=\frac{f(\xi,u)}{\rho(\xi)^a}\,\,\text{ in }\Omega,\,\, u|_{\partial\Omega}=0=\left.\frac{\partial u}{\partial \nu}\right|_{\partial\Omega},
\end{equation*}
where $0\in \Omega \subseteq \mathbb{H}$ is a bounded domain, $0\leq a\leq Q=4.$ The special feature of this problem is that it contains an
exponential nonlinearity and singular potential.
\end{abstract}
\maketitle
\tableofcontents
\section{Introduction}
In this article, we are interested to establish Adams' type inequality for biharmonic operator on Heisenberg group.
We also establish  Adams' type inequality with singular potential. As an application of Adams' type inequality,
we prove the existence of a solution to the following biharmonic equation with Dirichlet's boundary condition on Heisenberg group:
\begin{equation}\label{main_prob}
\begin{array}{cc}
\Delta_{\H}^2 u=\displaystyle\frac{f(\xi,u)}{\rho(\xi)^a} \,\,\,\,\mbox{ in }\Omega,\\ \\
 \displaystyle u|_{\partial\Omega}=0=\left.\frac{\partial u}{\partial \nu}\right|_{\partial\Omega},
\end{array}
\end{equation}
where $0\in \Omega $ is a bounded domain in one dimensional Heisenberg group $ \mathbb{H}$, $0\leq a<Q,\,Q=4$ is the homogeneous dimension
of $\mathbb{H}$ and $f:\Omega\times \R \rar \R$ satisfies either subcritical or critical exponential growth condition.
It is interesting to observe that in case of  $\Omega\subseteq \mathbb{H}^n,\, n\geq 2,$ by the Sobolev embedding theorem, the nonlinearity cannot exceed the
degree $\frac{2Q}{Q-4},$ while the Adams' inequality allows the nonlinearities to have exponential growth when $n=1.$ Therefore Adam's inequality motivates us to
discuss the above problem with exponential growth in $\Omega \subseteq \mathbb{H}.$

Problem \eqref{main_prob}, in bounded domains of  $\R^4$ has been discussed by A. C. Macedo \cite{costa}. Macedo established the existence of
a solution to the following problem  with the aid of singular version of Adams' inequality and by variational arguments:
\begin{equation}
\begin{array}{cc}
\Delta^2 u=\displaystyle\frac{f(x,u)}{|x|^a} \,\,\,\mbox{ in }\Omega,\\ \\
\displaystyle u|_{\partial\Omega}=0=\left.\frac{\partial u}{\partial \nu}\right|_{\partial\Omega},
\end{array}
\end{equation}
where $0\in \Omega \subseteq \R^4$ is a  bounded domain, $0\leq a<4$. M. de Souza \cite{desouza} established the existence of solution for
 the critical problem with singular potential $\displaystyle\frac{1}{|x|^a}$ in the case of $n$-Laplace operator in whole $\R^n$, using variational
techniques. J.M. do \'{O} et. al. \cite{desouza1} established the existence of a critical point to the following functional
\begin{equation}
J(u)=\frac{1}{n}\int_{\R^n}(|\nabla u|^n+|u|^n) dx -\int_{\R^n}\frac{F(u)}{|x|^a},
\end{equation}
where $n\geq 2,\, F:\R^n \rar \R$ is of class $C^1$ and $0\leq a<n.$ For the related works, see the references cited in \cite{desouza,desouza1,costa}.

For the  Trudinger-Moser type inequality in unbounded domains of $\R^2,$ and further generalizations  in unbounded domains in $\R^n,$ we refer to  \cite{rup_moser,li}.
For more details about Moser-Trudinger inequality, we refer to a survey by S.Y.A. Chang and P.C. Wang \cite{chang}.
Several existence results have been proved for problems involving Laplace and n-Laplace operator with exponential nonlinearities,
see for instance \cite{adi2,adi1,sami,fig,fig:survey,fre,bdo1,bdo2,ruf} and references cited therein.

Let us recall the developments on Trudinger-Moser inequality. Let $\Omega\subseteq \R^n,\, n\geq 2$ be a bounded domain.
The Sobolev embedding theorem says that for $p<n,$ $W_0^{1,p}(\Omega)\hookrightarrow L^q(\Omega),\, 1\leq q\leq \frac{np}{n-p}.$ For the limiting case $p=n,$ we have
\[W_0^{1,n}\hookrightarrow L^q(\Omega),\,\,1\leq q<\infty\]
but it is well known (see, Example 4.43 \cite{adamsbook}) that
\[W_0^{1,n}(\Omega) \not\hookrightarrow L^\infty(\Omega). \]
Then there is a natural question that what is the smallest  possible space in which, we have embedding of $W_0^{1,n}(\Omega)?$ This question was answered by N.S. Trudinger \cite{tru}. Trudinger proved that $W_0^{1,n}(\Omega)$ is embedded into Orlicz space $L_A(\Omega),$ where
\[A(t)=\exp\left(t^{\frac{p}{p-1}}\right)-1\]
is an $N$ function.
 Inequality by N.S. Trudinger \cite{tru},
which was later sharpened by J. Moser \cite{moser} is as follows:
\begin{theorem}
 Let $\Omega\subseteq \R^n$ be a bounded domain, $u\in W_0^{1,n}(\Om),\, n\geq 2$ and
 \begin{equation*}
 \int_\Om |\nabla u(x)|^n dx \leq 1,
 \end{equation*}
 then there exists a constant $C,$ which depends on $n$ only such that
 \begin{equation*}
 \int_\Om \exp(\alpha u^p) dx \leq C m(\Omega),
 \end{equation*}
 where
 \[p=\frac{n}{n-1},\, \alpha\leq \alpha_n =n\omega_n^{\frac{1}{n-1}},\, m(\Omega)=\int_\Omega dx\]
 and $\omega_{n-1}$ is the $(n-1)$-dimensional surface area of the unit sphere.

 The integral on the left actually is finite for any positive $\alpha,$ but if $\alpha>\alpha_n$ it can be made arbitrarily large by an appropriate choice of $u.$
\end{theorem}

In order to deal with problems involving higher order elliptic operators with exponential type nonlinearities, D.R. Adams \cite{adams}
extended the sharp inequality by J. Moser to higher order Sobolev spaces.
Adams proved the following:
\begin{theorem}
 Let $\Omega$ be a bounded and open subset of $\R^n.$ If $m$ is a positive integer less than $n,$ then there exists a constant
 $C_0=C(m,n)$ such that for all $u\in C^m(\R^n)$ with support contained in $\Omega$ and $\norm{\nabla^mu}_p \leq 1,\, p=\frac{n}{m},$ we have
 \[
\frac{1}{|\Omega|}\int_{\Omega}\exp(\beta|u(x)|^{\frac{n}{n-m}})dx\leq C_{0}
\]
for all $\beta\leq\beta(n,m)$ where
\[
\mathit{\beta(n,\ m)\ =}\left\{
\begin{array}
[c]{c}%
\frac{n}{w_{n-1}}\left[  \frac{\pi^{n/2}2^{m}\Gamma(\frac{m+1}{2})}%
{\Gamma(\frac{n-m+1}{2})}\right]  ^{p^\prime},\text{ }\mathrm{when}%
\,\,\,m\,\,\mathrm{is\,\,odd},\\
\frac{n}{w_{n-1}}\left[  \frac{\pi^{n/2}2^{m}\Gamma(\frac{m}{2})}{\Gamma
(\frac{n-m}{2})}\right]  ^{p^\prime},\text{ \ }\mathrm{when}%
\,\,\,m\,\,\mathrm{is\,\,even},%
\end{array}
\right.
\]
$p^\prime =\frac{p}{p-1}.$
Furthermore, for any  $\beta>\beta(n,m)$, the integral can be
made as large as desired, where
\[
\nabla^{m}u=\left\{
\begin{array}
[c]{l}%
\bigtriangleup^{\frac{m}{2}}u,\text{ \ \ \ \ \ \ }\mathrm{for}%
\,\,\,m\,\,\mathrm{even},\\
\nabla\bigtriangleup^{\frac{m-1}{2}}u,\,\,\ \mathrm{for}\,\,\,m\,\,\mathrm{odd}.%
\end{array}
\right.
\]
\end{theorem}
For applications of Adams' inequality to polyharmonic equations involving exponential type nonlinearities, we refer to \cite{bdo4,lak,lam1,mora}.
A  version of Moser-Trudinger inequality with singular potential was established by Adimurthi and K. Sandeep \cite{adi:sandeep}. They proved the following:
\begin{theorem}
 Let $\Omega$ be an open and bounded subset of $\R^n.$ Let $n\geq 2$ and $u\in W_0^{1,n}(\Omega).$ Then for every $\alpha>0$ and $\beta\in\left[\left.0,n\right)\right.,$
 \[\int_\Om \frac{\exp\left(\alpha |u|^{\frac{n}{n-1}}\right)}{|x|^\beta} dx < \infty.\]
 Moreover,
 \[\sup_{\norm{u}\leq 1} \int_\Om \frac{\exp\left(\alpha |u|^{\frac{n}{n-1}}\right)}{|x|^\beta} dx <\infty\]
 if and only if $\displaystyle \frac{\alpha}{\alpha_n}+\frac{\beta}{n} \leq 1,$ where $\norm{u}=\left(\int_\Om |\De u|^n\right)^{\frac{1}{n}}.$
 \end{theorem}

Motivated by this singular version of Moser-Trudinger inequality several authors studied the following problem
\begin{equation}\label{sin}
 \begin{array}{ll}
-\De_n u+ \lambda u|u|^{n-2}=\gamma \frac{f(x,u)}{|x|^\beta}+ kh(x,u) & \mbox{in }\Omega\subseteq \R^n,\\
u=0 & \mbox{on } \pa\Om,
 \end{array}
\end{equation}
in bounded as well as unbounded domains. See for instance, \cite{adi:sandeep,adi:yang,desouza,lam2} and references cited therein.
N. Lam and G. Lu \cite{lamlu}  established a version of singular Adams' inequality on bounded
domains. More precisely, they proved that:

\begin{theorem}\label{sing_adam}
Let $0\leq\alpha<n$ and $\Omega$ be a bounded domain in $\mathbb{R}^{n}$. Then for all $0\leq\beta\leq\beta_{\alpha,n,m}=\left(  1-\frac{\alpha
}{n}\right)  \beta(n,m)$, we have%
\begin{equation}
\underset{u\in W_{0}^{m,\frac{n}{m}}\left(  \Omega\right)  ,~\left\Vert
\nabla^{m} u\right\Vert _{\frac{n}{m}}\leq1}{\sup}\int_{\Omega}\frac
{e^{\beta\left\vert u\right\vert ^{\frac{n}{n-m}}}}{\left\vert x\right\vert
^{\alpha}}dx<\infty. \label{1.2}%
\end{equation}
When $\beta>\beta_{\alpha,n,m}$, the supremum is infinite. Moreover, when $m$
is an even number, the Sobolev space $W_{0}^{m,\frac{n}{m}}\left(
\Omega\right)  $ in the above supremum can be replaced by a larger Sobolev
space $W_{N}^{m,\frac{n}{m}}\left(  \Omega\right)  .$
\end{theorem}

In case of Heisenberg group $\h,$ W.S. Cohn and G. Lu \cite{cohn} established a Moser-Trudinger type inequality on bounded domains of $\h.$ They proved the following result:
\begin{theorem}
Let $\h$ be a n-dimensional Heisenberg group, $Q=2n+2,\, Q'=\frac{Q}{Q-1},$ and $\alpha_{Q}=\left(2\pi^n\Gamma(\frac{1}{2})\Gamma(\frac{Q-1}{2})\Gamma(\frac{Q}{2})^{-1}\Gamma(n)^{-1}\right)^{Q'-1}$. Then there exists a constant $C_0$ depending only on $Q$ such that for all $\Om\subseteq \h,$ $|\Om|<\infty,$
\be\label{heis_adam}
\underset{u\in W_0^{1,Q}(\Om),\norm{\nh u}_{L^Q(\Om)}\leq 1}{\sup}\frac{1}{|\Om|}\int_\Om \exp\left(\alpha_Q|u(\xi)|^{Q'}\right)d\xi  \leq C_0 <\infty.
\ee
If $\alpha_Q$ is replaced by any larger number, the integral in \eqref{heis_adam} is still finite for any $u\in W^{1,Q}(\h),$ but the supremum is infinite.
\end{theorem}
 Lam et. al. \cite{lam3} established the Moser-Trudinger type inequality with a singular potential. Their result reads as follows:
 \begin{theorem}
 Let $\h$ be a n-dimensional Heisenberg group, $\Om\subseteq \h,$ $|\Om|<\infty, \, Q=2n+2,\, Q'=\frac{Q}{Q-1},\, 0\leq \beta <Q,$ and $\displaystyle\alpha_Q=Q\sigma_{Q}^{\frac{1}{Q-1}},\, \sigma_Q=\int_{\rho(z,t)=1} |z|^Q d\mu.$ Then there exists a constant $C_0$ depending only on $Q$ and $\beta$ such that
 \be\label{singheis_adam}
\underset{u\in W_0^{1,Q}(\Om),\norm{\nh u}_{L^Q(\Om)}\leq 1}{\sup}\frac{1}{|\Om|^{1-\frac{\beta}{Q}}}\int_\Om \exp\left(\alpha_Q(1-\frac{\beta}{Q})|u(\xi)|^{Q'}\right)d\xi  \leq C_0 <\infty.
\ee
If $\alpha_Q\left(1-\frac{\beta}{Q}\right)$ is replaced by any larger number, then the supremum is infinite.
\end{theorem}

Motivated by the above research works, in order to obtain the existence of a solution to \eqref{main_prob} on Heisenberg group which involves exponential and singular nonlinearity,
it is natural to establish singular Adams type inequality on Heisenberg group. In fact, in this article, we first  establish Adams  type inequality for biharmonic operator
on Heisenberg group and also establish the singular  Adams type inequality. We, then prove existence of a solution to \eqref{main_prob} as an application to Adams type inequality,
where $f:\Om \times \R \rar \R$ is a function satisfying either subcritical or critical exponential growth condition.

We point out that very little research works are available for the existence of solution to singular elliptic equations on Heisenberg group even for the Laplacian,
see for instance \cite{tyagi_heisen,mok,chen}.  For existence results related to Laplace equation without singularity,
we refer to \cite{biagini,bir,bir1,bran,citti,gar,gamara,jia1,jia2,lan,lan1,lu,peng,ugu,ugu1}. For existence result concerning biharmonic operator on Heisenberg group,
we refer to \cite{zhang} and for qualitative questions related to biharmonic operator on Heisenberg group, we refer to \cite{GDJT}.

 \noindent Next, we define subcritical and critical growth for $f(\xi,u).$

\noindent We say that a function $f:\Om\times\R \rar \R$ has subcritical growth on $\Om\subseteq {\mathbb{H}}$ if
\be\label{subcrit}
\lim_{|u|\rar \infty} \frac{|f(\xi,u)|}{\exp(\alpha u^2)}=0,\,\,\text{uniformly on $\Om,$ }  \forall \,\alpha>0.
\ee
We say that $f$ has critical exponential growth if there exists $\alpha_0>$ such that
\be\label{crit1}
\lim_{|u|\rar \infty} \frac{|f(\xi,u)|}{\exp(\alpha u^2)}=0,\,\,\text{uniformly on $\Om,$ }  \forall \,\alpha>\al_0
\ee
and
\be\label{crit2}
\lim_{|u|\rar \infty} \frac{|f(\xi,u)|}{\exp(\alpha u^2)}=+\infty,\,\,\text{uniformly on $\Om,$ }  \forall \,\alpha<\al_0.
\ee

We define
\be\label{gam}
\Lambda=\underset{0\neq u\in D_0^{2,2}(\Om)}{\inf} \frac{\norm{u}^2}{\int_\Om \frac{|u|^2}{\rho(\xi)^a}}>0,
\ee
$\text{where }\xi=(z,t) \text{ and } \rho(\xi)=(|z|^4+t^2)^{\frac{1}{4}},\, 0\leq a <4.$

We assume the following conditions on the nonlinearity $f:$

\noindent(H1) $f:\bar{\Omega}\times \R \rar \R $ is continuous, $f(\xi,u)\geq 0$ on $\Omega \times \left[\left.0,\infty\right.\right),$ $f(\xi,u)\leq 0$ when $u\leq 0.$

\noindent(H2) There exists $R_0>0,\, M>0$ such that, $\forall \, u\geq R_0, \, \forall \xi\in \Omega $
\[0<F(\xi,u)\leq M f(\xi,u),\]
where $F(\xi,u)=\int_0^u f(\xi,s)ds.$

\noindent (H3) There exist $R_0>0,\, \theta >2$ such that $\forall \, |u|\geq R_0,\, \forall \xi\in \Om,$
\[\theta F(\xi,u) \leq uf(\xi,u).\]

\noindent(H4) $\displaystyle\limsup_{u\rar 0^+} \frac{2F(\xi,u)}{|u|^2}< \Lambda,$ where $\Lambda$ is defined by \eqref{gam}.

\noindent(H5) $\displaystyle\lim_{u\rar \infty}uf(\xi,u)\exp(-\al_0|u|^{2})\geq \beta_1 >\frac{(4-a)A }{4\al_0 R^{4-a} \mathcal{M}},$ where $
\mathcal{M}$ and $A$ are defined in  Section 2.

We remark that Problem \eqref{main_prob} has the following special features, which makes it challenging to study:
\begin{enumerate}[(i)]
\item It contains the nonlinearity $f,$ which is of exponential growth and potential $\displaystyle \frac{1}{\rho(\xi)^a},\, 0\leq a\leq 4,$ which has  singularity at $\xi=0.$
This problem is handled by the use of singular version of Adams' type inequality.
\item The case $a=4,$ is critical in the potential. Since we do not have the singular Adams' type inequality in  case of $a=4,$ therefore, we use the approximation method.
More precisely,  we approximate the Problem \eqref{main_prob} with a sequence of problems which are subcritical in potential, i.e. $a<4$ and then,
we pass the limit to conclude that Problem \eqref{main_prob} has a nontrivial solution in case $a=4.$
\end{enumerate}
 Next, we state our main results, which we will prove in next sections.
 \begin{theorem}\label{adams_ineq} Let $\Omega$ be a bounded domain in $\H.$
 Then there exists a constant $C(\Om)$ such that for all $u\in C_0^\infty(\Omega)$ and $\norm{\Delta_{\H}u}_{2}\leq 1,$
 \begin{equation}
 \int_\Om \exp\left(A|u(\xi)|^{2}\right)\leq C_0<\infty,
 \end{equation}
 where
 \[\displaystyle A=\frac{Q}{c_0\gamma_1^{2}},\]
 \begin{equation}\label{c0}
  c_0=\int_{\Sigma} d\mu,\, \Sigma=\{\xi\in \H^1:|\xi|=1\}
  \end{equation}
  and
 \begin{equation}\label{gamma}
 \gamma_1=\left(2\int_{\H^1}|z|^2(|z|^4+t^2+1)^{-\frac{5}{2}}d\xi\right)^{-1}.
 \end{equation} Furthermore, if we choose any number greater than $A$ then inequality  fails to hold.
  \end{theorem}

  \begin{theorem}\label{sing_adams} Let $\Omega$ be a bounded domain in $\H$ and $0\leq a <4.$  Then
there exists a constant $C_0(\Omega)$ such that for all $u\in C_0^\infty(\Omega)$ and $\norm{\Delta_{\H}u}_{2}\leq 1,$
 \[\int_\Om \frac{\exp\left(A\left(1-\frac{a}{4}\right)|u(\xi)|^{2}\right)}{\rho(\xi)^a}\leq C_0<\infty, \]
 where $\rho(\xi)=\sqrt{(|z|^4+t^2)^{\frac{1}{4}}},$ $\displaystyle A=\frac{4}{c_0\gamma_1^{2}}$ and $\gamma_1$ and $c_0$ are as defined by
\eqref{c0} and \eqref{gamma}, respectively. Furthermore, if we choose any number greater than $A\left(1-\frac{a}{4}\right)$ then inequality  fails to hold.
  \end{theorem}

\begin{theorem}\lab{mainthm_sub}
 Assume that $f$ satisfies the subcritical growth condition  \eqref{subcrit} and  (H1)-(H5) hold, then  Problem \eqref{main_prob} has a weak solution for $0<a<4$.
\end{theorem}
\begin{theorem}\lab{mainthm_crit}
 Assume that $f$ satisfies the critical growth condition  \eqref{crit1},\eqref{crit2} and  (H1)-(H5) hold, then  Problem \eqref{main_prob} has a weak solution for $0<a<4$.
\end{theorem}
We say \eqref{main_prob} has a critical potential case when $a=4.$ In this case there is no singular adams type inequality.
In   critical potential case, we establish the following:
\begin{theorem}\lab{mainthmapprox_sub}
 Assume that $f$ satisfies the subcritical growth condition  \eqref{subcrit} and  (H1)-(H5) hold, then  Problem \eqref{main_prob} has a weak solution for $a=4$.
\end{theorem}
\begin{theorem}\lab{mainthmapprox_crit}
 Assume that $f$ satisfies the critical growth condition  \eqref{crit1},\eqref{crit2} and  (H1)-(H5) hold, then Problem \eqref{main_prob} has a weak solution for $a=4$.
\end{theorem}

The plan of the article is as follows. In Section 2, we give important preliminaries on Heisenberg group and auxiliary results, which are used to prove the main theorems.
In Section 3, we prove Theorem \ref{adams_ineq} and Theorem \ref{sing_adams}. In Section 4, we prove Theorem \ref{mainthm_sub}--Theorem \ref{mainthmapprox_crit}.

\section{Preliminaries and Auxiliary Results}
First, let us recall the briefs on the Heisenberg group $\H^{n}.$ The Heisenberg group $\H^{n}= (\R^{2n+1},\,\cdot),$ is the space
$\R^{2n+1}$ with the non-commutative law of product
$$(x,y,t)\cdot(x',y',t')= (x+x',\,y+y', t+t'+ 2(\langle y,x'\rangle -\langle x,\,y'\rangle)),                     $$
where $x,\,y,\,x',\,y'\in \R^{n},\,t,\,t'\,\in \R$ and $\langle \cdot ,\cdot\rangle$ denotes the standard inner product in $\R^{n}.$
This operation endows $\H^{n}$ with the structure of a Lie group. The Lie algebra of $\H^{n}$
is generated by the left-invariant vector fields
$$ T= \frac{\partial}{\partial t},\,\,\,X_{i}= \frac{\partial}{\partial x_{i} }+ 2y_{i} \frac{\partial}{\partial t},\,\,
Y_{i}= \frac{\partial}{\partial y_{i} }- 2x_{i} \frac{\partial}{\partial t},\,\,i=1,\,2,.\,3,\,\ldots,\,n. $$
These generators satisfy the non-commutative formula
$$[X_{i},Y_{j}]= -4 \delta_{ij} T,\,\,[X_{i},X_{j}]= [Y_{i},Y_{j}]=  [X_{i},T] = [Y_{i},T] =0.$$

Let
$ z=(x,\,y)\in \R^{2n},\,\,\xi= (z,\,t)\in \H^{n}.$ The parabolic dilation $$\delta_{\la}\xi= (\la x,\,\la y,\,\la^{2} t)$$ satisfies
$$  \delta_{\la}(\xi_{0}\textbf{.}\xi     )  =  \delta_{\la}\xi   \textbf{.} \delta_{\la} \xi_{0}  $$
and $$ | \xi|  = (|z|^{4}+ t^{2} )^{\frac{1}{4} }= ((x^2 + y^2)^{2}+ t^{2} )^{\frac{1}{4} }  $$
is a norm with respect to the parabolic dilation which is known  as Kor\'{a}nyi gauge norm $N(z,\,t).$
In other words,
$\rho(\xi) = (|z|^{4}+ t^{2} )^{\frac{1}{4} }$
 denotes the Heisenberg distance between $\xi$ and the origin.
Similarly, one can define the distance between
$(z,\,t)$ and $(z',\,t')$ on $\H^{n}$ as follows:
$$\rho(z,\,t;z',t')= \rho((z',\,t')^{-1}\,.\,(z,\,t)).  $$
 It is clear that the vector fields $X_{i},\,Y_{i},\,i=1,\,2,\,\ldots,\,n$ are homogeneous of degree $1$ under the norm
$|\cdot\,|$ and $T$ is homogeneous of degree $2.$ The Lie algebra of Heisenberg group has the stratification
$\H^{n}= V_{1} \oplus V_{2},$ where the $2n$-dimensional horizontal space $V_{1}$
is spanned by $\{X_{i},\,Y_{i} \},\,\,i=1,\,2,\,\ldots,\,n,$ while $V_{2}$ is spanned by $T.$
The Kor\'{a}nyi ball of center $\xi_{0}$ and radius $r$ is defined by
$$  B_{\H^{n}}(\xi_{0},\,r)= \{ \xi: |\xi^{-1}\textbf{.} \xi_{0} | \leq r  \}                  $$
and it satisfies
$$ |B_{\H^{n}}(\xi_{0},\,r)| = |B_{\H^{n}}(0,\,r)| =  r^{d}|B_{\H^{n}}(\textbf{0},\,1)|,            $$
where $|.|$ is the $(2n+1)$-dimensional Lebesgue measure on $\H^{n}$  and $d= 2n+2$ is the so called the homogeneous dimension of Heisenberg group $\H^{n}.$
The Heisenberg gradient and Heisenberg Laplacian or the Laplacian-Kohn operator on
$\H^{n}$ are given by
$$\nabla_{\H^{n}}= (X_{1},\,X_{2},\,\ldots,\,X_{n},\, Y_{1},\,Y_{2},\,\ldots,\,Y_{n})$$
and
$$  \De_{\H^{n}}= \sum_{i=1}^{n} X_{i}^{2}+ Y_{i}^{2}=
\sum_{i=1}^{n}\left(\frac{\partial^{2} }{\partial x_{i}^{2}}+  \frac{\partial^{2} }{\partial y_{i}^{2}} +
4 y_{i} \frac{\partial^{2} }{\partial x_{i}\partial t}- 4 x_{i} \frac{\partial^{2} }{\partial y_{i}\partial t}
+ 4(x_{i}^{2}+y_{i}^{2} )\frac{\partial^{2} }{\partial t^{2}}   \right).$$

G.B. Folland \cite{folland} proved the existence of fundamental solution for the  sublaplacian $-\Delta_{\H^n}$ on the Heisenberg group $\H^n.$ Using Corollary 1 \cite{folland}, we have the following representation formula for each $u\in C_0^\infty(\Omega),$
\begin{equation}\label{rep}
u(\xi)= -\gamma_n \int_{\mathbb{H}^n}\Delta_{\H^n} u(\eta) |\xi\cdot \eta^{-1}|^{2-Q} d\eta ,
\end{equation}
 where $Q=2n+2$ is the homogeneous dimension of the Heisenberg group $\H^n$ and
 \begin{equation}\label{cn}
 \gamma_n=\left(n(n+1)\int_{\H^n}|z|^2(|z|^4+t^2+1)^{-\frac{n+4}{2}}d\xi\right)^{-1},\,\xi=(z,t).
 \end{equation}
 Next, we define convolution on $\H^n$, see \cite{folland_stein} for details.
 \begin{definition}[Convolution]
 If $f$ and $g$ are measurable functions on $\H^n,$ then their convolution $f\ast g$ is defined as
 \be\no
 (f\ast g)(\xi)=\int_{\H^n} f(\eta)g(\eta^{-1}\cdot \xi)d\eta=\int_{\H^n} f(\xi\cdot\eta^{-1})g(\eta)d\eta,
 \ee
provided the integrals converge.
\end{definition}

\begin{definition}[$D^{1,p}(\Omega)$ and $D_0^{1,p}(\Omega)$ Space]
 Let $\Omega \subseteq \h$ be open and $1<p<\infty$. Then we define
 \[D^{1,p}(\Omega)=\{u:\Omega\rightarrow \R\text{ such that }\, u,|\nh u| \in L^p(\Omega) \}.\]
 $D^{1,p}(\Omega)$ is equipped with the norm
 \[\norm{u}_{D^{1,p}(\Omega)}=\left(\norm{u}_{L^p(\Omega)}+\norm{\nh u}_{L^p(\Omega)}\right)^{\frac{1}{p}}.\]
 $D_0^{1,p}(\Omega)$ is the closure of $C_0^\infty(\Omega)$ with respect to the norm
 \[\norm{u}_{D_0^{1,p}(\Omega)}=\left(\int_\Omega |\nh u|^p dz dt\right)^{\frac{1}{p}}.\]
\end{definition}
\begin{definition}[$D^{2,p}(\Omega)$ and $D_0^{2,p}(\Omega)$ Space]
 Let $\Omega \subseteq \h$ be open and $1<p<\infty$. Then we define
 \[D^{2,p}(\Omega)=\{u:\Omega\rightarrow \R\text{ such that }\, u,|\nh u|, |\Dh u| \in L^p(\Omega) \}.\]
 $D^{2,p}(\Omega)$ is equipped with the norm
 \[\norm{u}_{D^{2,p}(\Omega)}=\left(\norm{u}_{L^p(\Omega)}+\norm{\nh u}_{L^p(\Omega)}+\norm{\Dh u}^p\right)^{\frac{1}{p}}.\]
 $D_0^{2,p}(\Omega)$ is the closure of $C_0^\infty(\Omega)$ with respect to the norm
 \[\norm{u}_{D_0^{2,p}(\Omega)}=\left(\int_\Omega |\Dh u|^p dz dt\right)^{\frac{1}{p}}.\]
\end{definition}

\begin{theorem}(Embedding Theorem)\label{embedding}
 Let $k\in \mathbb{N}$ and $p\in\left[\left.1,\infty\right.\right).$
 \begin{enumerate}[(i)]
\item \ If $k<\frac{Q}{p},$ then $ D_0^{k,p}(\Omega)$ is continuously embedded into $L^{p^*}(\Omega),$\\ for $\displaystyle \frac{1}{p^*}=\frac{1}{p}-\frac{k}{Q}.$
 \item If $k=\frac{Q}{p},$ then  $ D_0^{k,p}(\Omega)$ is continuously embedded into $L^r(\Omega),$ for $r\in \left[\left.1,\infty\right.\right).$
\item If $k>\frac{Q}{p},$ then $\displaystyle D_0^{k,p}(\Omega)$ is continuously embedded into  $C^{0,\gamma}(\bar{\Omega}),$ for all $0\leq \gamma < k-\frac{Q}{p}.$
\end{enumerate}
\end{theorem}

Now, we define the Adams functions. Let $B:=B(0,1)$ denote the unit ball in $\H^4$ and $B_{\ell}=B(0,\ell)$ denotes the ball with center $0$ and radius $\ell.$
We have the following result.
\begin{lemma}\cite{lak}
For all $\ell\in (0,1)$ there exists $U_\ell \in D:=\{u\in D_0^{2,2} (B): u\mid_{B_\ell}=1\},$
such that
\[\norm{U_\ell}=C(B_\ell, B)\leq \frac{A }{Q \log\left(\frac{1}{\ell}\right)},\]
where $Q=2n+2$ is homogeneous dimension of Heisenberg group $\H^n$ and  $C(K,E)$ denote the conductor capacity of $K$ in $E,$ whenever $E$ is an open set and $K$ a relatively compact subset of $E,$ which is defined as follows:
\[C(K,E):= \inf\{\norm{\Delta_{\mathbb{H}^n} u}_2^2: u\in C_0^\infty(E), \, u\mid _{K} =1\}.\]
\end{lemma}
Let $0\in \Om$ and $R\leq \text{dist}(0,\pa\Om)$, the Adams function is defined as follows:
\be\label{adam_func}
\tilde{A}_r(\xi)=\left\{
                      \begin{array}{ll}
                        \sqrt{\frac{Q \log\left(\frac{R}{r}\right)}{A }}U_{r/R}\left(\frac{\xi}{R}\right), & {|\xi|<R;} \\
                        0, & {|\xi|\geq R,}
                      \end{array}
                    \right.
 \ee
where $0<r<R.$
It is easy to check that $\norm{\tilde{A}_r}\leq 1$ and we denote
\[\mathcal{M} =\lim_{k\rar\infty}\int_{\frac{1}{k}\leq |\xi|\leq 1}\exp \left(Q\log k |U_{R/k}(\xi)| )d\xi.\right.\]
We have $\mathcal{M}>0,$ for the details, we refer to \cite{lak}.

Next, we recall decreasing rearrangement of functions on Heisenberg group. For the details about rearrangement on Heisenberg group, we refer to \cite{folland_stein}. Let $\Omega $ be a bounded and measurable subset of $\H^n.$ Let $f:\Omega \rar \R$ be a measurable function. For $t\in \R,$ the level set $\{f>t\}$ is defined as
\[\{f>t\}=\{\xi\in\Om : f(\xi)>t\}.\]
Sets $\{f<t\},\,\{f\geq t\}$ and $\{f=t\}$ can be defined in an analogous way.
\begin{definition}[Distribution Function]
Let $f:\Omega \rar \R$ be a measurable function then distribution function of $f$ is given by
\[\lambda_{f}(t)=|\{f>t\}|,\]
where $|A|$ denotes the Lebesgue measure of the set $A.$
\end{definition}
It is easy to see that distribution function is a monotonically decreasing function of $t$ and
\[\lambda_f(t)=\begin{cases}
0, & t\geq ess\sup(f),\\
|\Om|, & t\leq ess\inf(f).
\end{cases}
\]

 Thus the range of $\lambda_f$ is the interval $[0\,, |\Om|].$
\begin{definition}[Decreasing Rearrangement]
Let $\Omega\subset \H^n$ be bounded and let $f:\Omega \rar \R$ be a measurable function. Then the decreasing rearrangement of $f$ is defined as
\[f^\ast (0)=ess\sup(f),\]
\[f^\ast(s) =\inf\{t:\lambda_f(t)<s\},\,\, s>0.\]
\end{definition}
\begin{lemma}\label{rearrange1}
Let $\Omega\subset \H^n$ be bounded and let $f:\Omega \rar \R$ be a measurable function. Then for $0<p<\infty,$
\[\int_\Om |f(\xi)|^p d\xi=\int_0^{|\Om|} |f^\ast(t)|^p dt.\]
\end{lemma}
\begin{proof} For a proof, we refer to Chapter 1 \cite{folland_stein}.
\end{proof}
\begin{lemma}[Hardy-Littlewood inequality]\label{hardy_littlewood}
Let $\Omega\subset \H^n$ be bounded and let $f,g:\Omega \rar \R$ be a measurable functions. Then
\[\int_\Om |f(\xi)g(\xi)|d\xi\leq \int_0^{|\Om|} f^\ast(t)g^\ast(t)dt.\]
\end{lemma}
\begin{proof} For a proof, we refer to Chapter 1 \cite{folland_stein}.
\end{proof}

The function $f^{\ast\ast}$ on $(0,\infty)$ is defined as
\[f^{\ast\ast}(t)=\frac{1}{t}\int_0^\infty f^\ast(s)ds.\]
Next, we state Vitali's convergence theorem. We refer to \cite{rudin} for the proof.
\begin{theorem}[Vitali's convergence theorem] Let $(X,\mathcal{F},\mu)$ be a measure space such that $\mu(X)<\infty.$ Suppose
\begin{enumerate}[(i)]
 \item $\{f_n\}$ is uniformly integrable,
 \item $f_n(x) \rar f(x) \text{ a.e. as } n\rar\infty$,
 \item $|f(x)|<\infty,$ a.e. in $X,$
\end{enumerate}
then $f\in L^1(X,\mu)$ and
\[\lim_{n\rar\infty} \int_X |f_n-f| d\mu =0.\]
\end{theorem}
\begin{theorem}[Converse of Vitali's theorem]
Let $(X,\mathcal{F},\mu)$ be a measure space such that $\mu(X)<\infty.$ Let $f_n\in L^1(X,\mu)$ and
\[\lim_{n\rar \infty} \int_E f_n d\mu\]
exists for every $E\in \mathcal{F},$ then $\{f_n\}$ is uniformly integrable.
\end{theorem}

Let \[J: D_0^{2,2}(\Omega) \longrightarrow \R\] be a functional defined by
\be\label{func}
J(u)=\frac{1}{2} \int_\Om |\Delta_{\mathbb{H}^n} u|^2 dx -\int_\Om \frac{F(\xi,u)}{\rho(\xi)^a} dx ,
\ee
where $\displaystyle F(\xi,u)=\int_0^u f(\xi,s) ds. $
Throughout this article, we denote $\|{\cdot}\|_{D_0^{2,2}(\Omega)}$ by $\|{\cdot}\|$ and $||\cdot||_p$ denotes the standard $L^p$-norm.

 \section{Proof of Theorem \ref{adams_ineq} and Theorem \ref{sing_adams}}
 In order to prove Theorem \ref{adams_ineq} and Theorem \ref{sing_adams}, we need the following results. In this paper $C$ is some generic constant which may vary from line to line.
 W.S. Kohn and G. Lu \cite{cohn} proved the following theorem:
\begin{theorem}\label{kohn_thm}
Let $\Omega\subseteq \H^n$ be bounded domain and $Q=2n+2$ be a homogeneous dimension of $\H^n.$ Let $0<\alpha<Q,$ $Q-\alpha p=0,$ $\displaystyle p'=\frac{Q}{Q-\alpha}$ and 
\begin{equation}\label{ialpha}
(I_\alpha \ast f)(\xi)=\int_{\H^n} |\xi\cdot\eta^{-1}|^{\alpha-Q} f(\eta) d\eta.
\end{equation}
Then there exists a constant $C$ such that for all $f\in L^p(\H^n)$ with support in $\Omega,$
\[\frac{1}{|\Omega|}\int_\Omega \exp\left(\frac{Q}{c_0}\left|\frac{(I_\alpha\ast f)(\xi)}{\norm{f}_{L^p(\H^n)}}\right|^{p'}\right)d\xi\leq  C,\]
where $\displaystyle c_0=\int_{\Sigma} d\mu,\, \Sigma=\{\xi\in \H^n:|\xi|=1\}.$ Furthermore, if $Q/c_0$ is replaced by a greater number, then the statement is false.
 \end{theorem}
 In particular, for $\alpha=2$ and $n=1,$ we get the following corollary:
 \begin{cor}\label{i2exp}
 There exists a constant $C$ such that for all $\Omega\subseteq \H,\, |\Omega|<\infty,$ and for all $f\in L^2(\H)$ with support in $\Omega,$
\[\frac{1}{|\Omega|}\int_\Omega \exp\left(\frac{4}{c_0}\left|\frac{(I_2\ast f)(\xi)}{\norm{f}_{{L^2}(\H)}}\right|^{2}\right)d\xi\leq  C,\]
where $\displaystyle c_0=\int_{\Sigma} d\mu,\, \Sigma=\{\xi\in \H:|\xi|=1\}.$ Furthermore, if $4/c_0$ is replaced by a greater number, then the statement is false.
 \end{cor}

  \begin{lemma}\label{ast}
  Let $0<\alpha<1,\, 1<p<\infty$ and $b(s,t)$ be a non-negative measurable function  on $(-\infty,\infty)\times \left[\left.0,\infty\right)\right.$ such that almost everywhere,
  \[b(s,t)\leq 1,\,\, \text{when }0<s<t,\]
  \[\sup_{t>0}\left(\int_{-\infty}^0+\int_t^\infty b(s,t)^{p'} ds\right)^{\frac{1}{p'}}=b<\infty.\]
  Then there is a constant $C(p,\alpha)$ such that if for $\phi\geq 0$
  \[\int_{-\infty}^\infty \phi(s)^p ds\leq 1,\]
  then
  \[\int_0^\infty \exp(-F_\alpha(t))dt \leq C,\]
  where
  \[F_\alpha(t)=\alpha t-\alpha\left(\int_{-\infty}^\infty b(s,t)\phi(s) ds\right)^{p'}.\]
  \end{lemma}
  \begin{proof}
  In  case of $\alpha=1,$ this lemma was proved by D.R. Adams \cite{adams}, which was later modified for the case $0<\alpha\leq1$ by N. Lam and G. Lu \cite{lamlu}. We refer to \cite{adams,lamlu} for the details.
  \end{proof}
  Let $U=f\ast g$ denote the convolution on $\h.$ Then O'Neil \cite{oneil} proved the following lemma:
  \begin{lemma}\label{reconv}
  \[U^\ast(t)\leq U^{\ast\ast}(t)\leq tf^{\ast\ast}(t)g^{\ast\ast}(t)+\int_t^\infty f^\ast(s)g^\ast(s)ds.\]
  \end{lemma}
  Now, we are ready to prove Theorem \ref{adams_ineq}.

\noindent\textbf{Proof of Theorem \ref{adams_ineq}:}
Using  \eqref{rep}, we get
\begin{align}\no
|u(\xi)|&\leq \gamma_1\int_\Omega \De_{\H}u(\eta)|\xi\cdot\eta^{-1}|^{-2} d\eta\\ \no
& \leq \gamma_1 |(I_2\ast \Delta_{\H} u)(\xi)|\,\,(\text{by \eqref{ialpha} with $\alpha=2$})\\ \label{adam1}
|u(\xi)|^{2} &\leq  \gamma_1^{2} |(I_2\ast \Delta_{\H} u)(\xi)|^{2}.
\end{align}
Using Corollary \ref{i2exp} and Equation \eqref{adam1}, we get
 \[\int_\Om \exp\left(A|u(\xi)|^{2}\right) d\xi \leq \int_\Omega \exp\left(A\gamma_1^{2} |(I_2\ast \Delta_{\H} u)(\xi)|^{2}\right)\leq C_0, \]
 provided $\displaystyle A \gamma_1^{2}\leq \frac{4}{c_0},$ i.e. , $\displaystyle A\leq \frac{4}{c_0\gamma_1^{2}}. $
 This completes the first part of the  proof.

 The proof of sharpness of the constant has similar lines  as pp. 393 \cite{adams}, so we omit the details.
 \qed

In order to prove Theorem \ref{sing_adams}, first we prove  auxiliary lemmas, which are used in the proof.

\begin{lemma}\lab{gxi} Let $Q=2n+2$ be homogeneous dimension of $n$-dimensional Heisenberg group $\H^n$ and $g(\xi)=\rho(\xi)^{2-Q},$ then
\begin{equation}\nonumber
 g^\ast(t)=\left(\frac{c_0}{Qt}\right)^{\frac{1}{p'}},
\end{equation}
and
\begin{equation}\nonumber
 g^{\ast\ast}(t)=pg^\ast(t),
\end{equation}
where $\displaystyle\rho(\xi)=|\xi|=(|z|^4+t^2)^{\frac{1}{4}},\, p=\frac{Q}{2},\, p'=\frac{Q}{Q-2}$ and  $c_0$ is defined in \eqref{c0}.
\end{lemma}
\begin{proof}
We have \[
g^\ast(t)=\inf\{s>0: \lambda_g(s)\leq t\},\]
 where
 \[\lambda_g(s)=|\{\xi\in \Omega: g(\xi)>s\}|.\]
 Now,
 \begin{align}\no
 |\{\xi\in \Omega: g(\xi)>s\}|&=|\{\xi\in \Omega: |\xi|^{2-Q}>s\}|\\ \label{meas}
 &= |\{\xi\in \Omega: |\xi|<s^{-\frac{1}{Q-2}}\}|
 \end{align}
 By using polar coordinates (Proposition 1.15 \cite{folland_stein}), from \eqref{meas}, we obtain
 \begin{align}\no
 \lambda_g(s)&=\int_\Sigma \int_0^{s^{-\frac{1}{Q-2}}} r^{Q-1} dr d\mu, \text{ where $\Sigma$ is defined in \eqref{c0}}\\ \label{ld}
 &= \frac{c_0}{Q}s^{-\frac{Q}{Q-2}}.
 \end{align}
 From \eqref{ld}, we see that,  for any $t>0,$
 \begin{align}\no
 \lambda_g(s)<t &\Rightarrow \frac{c_0}{Q}s^{-\frac{Q}{Q-2}}< t\\ \no
 & \Rightarrow s^{-\frac{Q}{Q-2}} < \frac{Q}{c_0}t \\ \label{ld1}
 & \Rightarrow s> \left(\frac{c_0}{Qt}\right)^{\frac{Q-2}{Q}}=\left(\frac{c_0}{Qt}\right)^{\frac{1}{p'}}
 \end{align}
 From \eqref{ld1}, we obtain
 \be\label{ld2}
g^\ast(t) \geq \left(\frac{c_0}{Qt}\right)^{\frac{1}{p'}}.
\ee
Now, for $\displaystyle s=\left(\frac{c_0}{Qt}\right)^{\frac{1}{p'}},$
\begin{equation}\label{ld3}
\lambda_g(s)=t.
\end{equation}
From \eqref{ld3}, we obtain
\begin{equation}\label{ld4}
g^\ast (t) \leq \left(\frac{c_0}{Qt}\right)^{\frac{1}{p'}}.
\end{equation}
From \eqref{ld2} and \eqref{ld4}, we conclude that
\[g^\ast(t)=\left(\frac{c_0}{Qt}\right)^{\frac{1}{p'}}.\]
Next, we compute $g^{\ast\ast}(t).$
\begin{align*}
g^{\ast\ast}(t)&=\frac{1}{t}\int_0^t g^\ast (s) ds\\
&=\frac{1}{t}\int_0^t \left(\frac{c_0}{Qs}\right)^{\frac{1}{p'}}ds \\
&=\frac{1}{t}\left(\frac{c_0}{Q}\right)^{\frac{1}{p'}}\int_0^t s^{-\frac{1}{p'}} ds\\
&= p\frac{1}{t} \left(\frac{c_0}{Q}\right)^{\frac{1}{p'}} t^{\frac{1}{p}}\\
&=pg^\ast(t).
\end{align*}
This completes the proof.

\end{proof}
  \begin{lemma}\label{sing_conv}
   Let $\Omega\subseteq \H, $ be a bounded domain, and $(I_2\ast f)(\xi)=\int_{\H} |\xi\cdot\eta^{-1}|^{-2} f(\eta) d\eta$. Then there exists a constant $C>0$ such that for all  $f\in L^2(\H)$ with support in $\Omega,$
\[\frac{1}{|\Omega|}\int_\Omega \frac{\exp\left(\frac{4}{c_0}\left(1-\frac{a}{4}\right)\left|\frac{(I_2*f)(\xi)}{\norm{f}_{{L^2}(\H)}}\right|^{2}\right)}{\rho(\xi)^a}\leq  C,\]
where $\displaystyle c_0=\int_{\Sigma} d\mu,\, \Sigma=\{\xi\in \H^n:|\xi|=1\}.$ Furthermore, if $\frac{4}{c_0}\left(1-\frac{a}{4}\right)$ is replaced by a greater number, then the statement no longer holds.
\end{lemma}
\begin{proof}Let \[u(\xi)=(g\ast f)(\xi), \text{ where } \]

\[g(\xi)=\rho(\xi)^{-2}.\]
Then by definition
\[u(\xi)=(I_2\ast f)(\xi)\]

 \noindent and by Lemma \ref{gxi}(with Q=4), we get
\begin{equation}\label{pf1}
 g^\ast(t)=\left(\frac{c_0}{4t}\right)^{\frac{1}{2}},\,
 g^{\ast\ast}(t)=2g^\ast(t).
\end{equation}
By Lemma \ref{reconv}, we get
\begin{align}\no
u^\ast(t)&\leq u^{\ast\ast}(t)\leq tf^{\ast\ast}(t)g^{\ast\ast}(t)+\int_t^{|\Om|} f^\ast(s)g^\ast(s) ds \\ \no
&=t.\frac{1}{t}2g^\ast(t)\int_0^t f^\ast(s)ds+\int_t^{|\Om|} f^\ast(s) \left(\frac{c_0}{4}\right)^{\frac{1}{2}} s^{-\frac{1}{2}}ds\,\, (\text{by \eqref{pf1}})\\ \label{sc1}
& =\left(\frac{c_0}{4}\right)^{\frac{1}{2}}\left(2t^{-\frac{1}{2}}\int_0^tf^\ast(s)ds+\int_t^{|\Om|} s^{-\frac{1}{2}}f^\ast(s) ds\right).
\end{align}
Now, using the change of variables,
\begin{equation}\label{change}
\phi(s)=|\Om|^\frac{1}{2}f^\ast(|\Om|e^{-s})e^{-\frac{s}{2}},
\end{equation}
we get
\begin{align}\no
\int_\Omega (f(x))^2 dx&=\int_0^{|\Om|}(f^\ast(t))^2 dt\\
& =\int_0^\infty (\phi(s))^2 ds.
\end{align}
Let $h(\xi)=\frac{1}{\rho(\xi)},$ then $h^\ast(\xi)=\frac{V}{t}^{\frac{a}{4}},$ where $V$ is volume of unit ball in $\H.$\\
By the Hardy-Littlewood inequality (Lemma \ref{hardy_littlewood}), we obtain
\be\label{hl}
\int_\Om \frac{\exp\left(\left(1-\frac{a}{4}\right)\frac{4}{c_0}|u(\xi)|^{2}\right)}{\rho(\xi)^a} d\xi\leq (V)^{\frac{a}{4}}\int_0^{|\Om|} \frac{\exp{\left(\left(1-\frac{a}{4}\right)\frac{4}{c_0}(u^\ast(t))^{2}\right)}}{t^{\frac{a}{4}}}.
\ee
Let us introduce the change of variable
\[t=|\Omega|e^{-s},\, \text{then } dt=-|\Omega|e^{-s} ds\]
and using this change of variable, we get
\begin{align}\no
&(V)^{\frac{a}{4}}\int_0^{|\Om|} \frac{\exp{\left(\left(1-\frac{a}{4}\right)\frac{4}{c_0}(u^\ast(t))^{2}\right)}}
{t^{\frac{a}{4}}} dt\\ \no
 &=(V)^{\frac{a}{4}}\int_0^\infty\frac{\exp\left(1-\frac{a}{4}\right)\frac{4}{c_0}(u^\ast(|\Om|e^{-s}))^{2}}{(|\Om|e^{-s})^{\frac{a}{4}}} |\Om|e^{-s} ds\\ \no
&\leq(V)^{\frac{a}{4}}|\Om|^{1-\frac{a}{4}} \int_0^\infty \exp\left[\left(1-\frac{a}{4}\right)\left\{p(|\Om| e^{-s})^{-\frac{1}{2}}\int_0^{|\Om|e^{-s}}f^\ast(z) dz+\right.\right.
\\ \no
& \left.\left.\int_{|\Om|e^{-s}}^{|\Om|}f^\ast(z)z^{-\frac{1}{2}} dz\right\}^{2}-\left(1-\frac{a}{4}\right)s\right]ds \,\, (\text{by\,\eqref{sc1}})\\ \no
&=(V)^{\frac{a}{4}}|\Om|^{1-\frac{a}{4}} \int_0^\infty \exp\left[\left(1-\frac{a}{4}\right)\left\{pe^{\frac{s}{2}}\int_s^\infty \phi(w)e^{-\frac{w}{2}} dw+\int_0^s \phi(w)dw\right\}^{2}\right.\\ \no
&\left.-\left(1-\frac{a}{4}\right)s\right]ds\,\,(\text{by using the value of $f^\ast(z)$ from \eqref{change}})\\
&=(V)^{\frac{a}{4}}|\Om|^{1-\frac{a}{4}} \int_0^\infty \exp\left[-F_{\left(1-\frac{a}{4}\right)}(s)\right] ds,
\end{align}
where $F_{1-\frac{a}{4}}(s)$ is as in Lemma \ref{ast} with
\[b(s,t)=\begin{cases}
0 & -\infty <s\leq 0,\\
1 & 0<s<t,\\
2e^{\frac{t-s}{2}} & t<s<\infty.
\end{cases}\]
Since $u(\xi)=(I_2\ast f)(\xi),$ therefore in view of \eqref{hl}, it is enough to show that
\be\label{fr}
\int_0^\infty \phi(s)^2 ds \leq 1 \text{ implies }  \int_0^\infty \exp\left(-F_{\left(1-\frac{a}{4}\right)}(s)\right)ds \leq C.
\ee
\eqref{fr} follows by using Lemma \ref{ast}. This completes the proof.
\end{proof}
\noindent\textbf{Proof of Theorem \ref{sing_adams}:}
Using the Formula \eqref{rep}, we get
\begin{align}\no
|u(\xi)|&\leq \gamma_1\int_\Omega \De_{\H}u(\xi)|\xi\cdot\eta^{-1}|^{-2} d\eta\\ \no
& \leq \gamma_1 |(I_2\ast \Delta_{\H} u)(\xi)|\,\,(\text{by \eqref{ialpha} with $\alpha=2$})\\ \label{adam2}
|u(\xi)|^{2} &\leq  \gamma_1^{2} |(I_2\ast \Delta{\H} u)(\xi)|^{2}.
\end{align}
Using Lemma \ref{sing_conv} and \eqref{adam2}, we get
 \[\int_\Om \frac{\exp\left(A\left(1-\frac{a}{4}\right)|u(\xi)|^{2}\right)}{\rho(\xi)^a}\leq \int_\Om \frac{\exp\left(A\left(1-\frac{a}{4}\right)\gamma_1^{2} |(I_2\ast \Delta_{\H} u)(\xi)|^{2}\right)}{\rho(\xi)^a}\leq C_0, \]
 provided $\displaystyle A\left(1-\frac{a}{4}\right)\gamma_1^{2} \leq \frac{4}{c_0}.$

For the  sharpness of the constant, we refer to \cite{adams}.
 This completes the  proof.
\qed

\section{Proof of Theorems \ref{mainthm_sub}-\ref{mainthmapprox_crit} }
In order to prove Theorem 1.9-1.12, we obtain mountain pass geometry of the associated functional. The following lemmas deal with the geometric requirements of mountain pass theorem.

\begin{lemma}\label{geo1}
 Assume that $f$ satisfies \eqref{subcrit} and suppose (H1)-(H5) hold. Then there exists $\rho>0 $ such that
 \[J(u)>0,\, \, \text{if } \norm{u}=\rho.\]
\end{lemma}
\begin{proof}
 By (H4), we have that
 \[\limsup_{s\rar 0^+} \frac{2F(\xi,s)}{|s|^2}< \Lambda,\]
 which by definition is same as
 \be\label{mpg1}
 \inf_{\beta>0}\sup \left\{\frac{2F(\xi,s)}{|s|^2}: 0<s<\beta\right\}<\Lambda.
 \ee
 Since \eqref{mpg1} is strict inequality, therefore, we can choose a number $\tau>0$ such that
 \be\label{mpg2}
 \inf_{\beta>0}\sup \left\{\frac{2F(\xi,s)}{|s|^2}: 0<s<\beta\right\}<\Lambda-\tau.
 \ee
 Since in \eqref{mpg2} infimum is strictly less than $\Lambda-\tau,$ therefore there exists $\delta>0$ such that
 \be
 \sup\left\{\frac{2F(\xi,s)}{|s|^2}: 0<s<\delta\right\}<\Lambda-\tau.
\ee
Thus for $|s|<\delta$
\[\frac{2F(\xi,s)}{|s|^2} < \Lambda-\tau,\]
or
\be\label{mpg3}
F(\xi,s) < \frac{1}{2} (\Lambda-\tau )|s|^2.
\ee
Since $f$ has  subcritical exponential growth therefore there exist constants $C>0$ and
$\gamma>0$ such that
\be\label{growth}
|f(\xi,t)|\leq C \exp(\gamma t^{2}),\, \forall \xi \in\Om,\,\forall t\in \R.
\ee
Thus we have
\begin{align}\nonumber
|F(\xi,s)|&=\left|\int_0^s f(\xi,t) dt \right|\\ \no
&\leq \int_0^s |f(\xi,t)| dt\\ \no
&\leq C\int_0^s  \exp(\gamma t^{2}) dt\,\,\,(\text{by \eqref{growth}})\\ \label{mpg4}
&\leq C\exp(\gamma s^{2}).
\end{align}
Now for $|s|\geq \delta $ and $q> 2,$ there exists a constant $K(\delta, q)$ such that
\be\label{mpg5}
|F(\xi,s)|\leq K |s|^q \exp(\gamma s^{2}),\,\,\,\,\,\forall \,|s|\geq \delta.
\ee
On using \eqref{mpg3} and \eqref{mpg4}, we get
\be\label{mpg6}
F(\xi,s) \leq \frac{1}{2} (\Lambda-\tau)|s|^2+K \exp(\gamma |s|^{2}) |s|^q,
\ee
for all $\xi\in \Om,\, s\in \R$ and for some $\gamma, \tau>0$ and $q>2.$\\
Now consider $r$ and $r^\prime$ such that $\displaystyle \frac{1}{r}+\frac{1}{r^\prime}=1,$ then by H\"{o}lder's inequality, we have
\begin{align}\no
\int_\Om \frac{\exp(\gamma |u|^{2})|u|^q}{\rho(\xi)^a} d\xi &\leq
\left(\int_\Om \frac{\exp(\gamma r |u|^{2})}{\rho(\xi)^{ar}} dx\right)^{\frac{1}{r}}\left(\int_\Om |u|^{qr^\prime} d\xi\right)^{\frac{1}{r^\prime}}\\  \label{mpg7}
&\leq \left(\int_\Omega \frac{\exp\left(\gamma r\norm{u}^{2}
\left(\frac{|u|}{\norm{u}}\right)^{2}\right)}{\rho(\xi)^{ar}}\right)^{\frac{1}{r}}\left(\int_\Om |u|^{qr^\prime}d\xi\right)^{\frac{1}{r^\prime}}.
\end{align}
Now, if we choose $r>1$ sufficiently close to $1,$ so that $ar<4$ and $\norm{u}\leq\sigma$ such that $\gamma r\sigma^2<A\left(1-\frac{a}{4}\right).$
Then by Theorem \ref{sing_adams} and \eqref{mpg7}, we get
\be
\int_\Om \frac{\exp(\gamma |u|^{2})|u|^q}{\rho(\xi)^a} dx \leq C \left(\int_\Om |u|^{qr^\prime}dx\right)^{\frac{1}{r^\prime}}.
\ee
Therefore, we get
\be\label{mpg8}
J(u) \geq \frac{1}{2}\norm{u}^2-\frac{\Lambda-\tau}{2}\int_\Om \frac{|u|^2}{\rho(\xi)^a}dx -C \left(\int_\Om |u|^{r^\prime q}\right)^{\frac{1}{r^\prime}}.
\ee
Now, we have
\be\label{eigen}
\Lambda=\underset{0\neq u\in D_0^{2,2} (\Om)}{\inf} \frac{\norm{u}^2}{\int_\Om \frac{|u|^2}{\rho(\xi)^a}}.
\ee
\eqref{eigen} implies that
\[\Lambda\leq  \frac{\norm{u}^2}{\int_\Om \frac{|u|^2}{\rho(\xi)^a}}\,\,\forall\, 0\neq u\in D_0^{2,2} (\Om)\]
or
\be\label{eigen1}
 \int_\Om \frac{|u|^2}{\rho(\xi)^a} \leq \frac{1}{\Lambda}\norm{u}^2.
 \ee
 On using \eqref{eigen1} in \eqref{mpg8}, we get
 \be\label{mpg9}
 J(u) \geq \frac{1}{2} \norm{u}^2-\frac{\Lambda-\tau}{2\Lambda}\norm{u}^2 - C\norm{u}_{r^\prime q}^q.
 \ee
 Since by  Theorem \ref{embedding}, $D_0^{2,2} (\Om)$ is continuously embedded into $L^s(\Omega),$ for all $1\leq s<\infty.$ Therefore, in particular, for $s=r^\prime q,$ we get
 \be\label{mpg10}
 \norm{u}_{r^\prime q}\leq C\norm{u}.
 \ee
 On using \eqref{mpg10} in \eqref{mpg9}, we get
 \begin{align*}
 J(u) &\geq \frac{1}{2} \left(1-\frac{\Lambda-\tau}{\Lambda}\right)\norm{u}^2-C \norm{u}^q.
 \end{align*}
 Since $\tau>0$ and $q>2,$ choose $\rho>0$ such that
 \[\frac{1}{2} \left(1-\frac{\Lambda-\tau}{\Lambda}\right)\rho-C\rho^{q-1}>0,\]
 then, we have
 \[J(u)\geq \norm{u} \left( \frac{1}{2} \left(1-\frac{\Lambda-\tau}{\Lambda}\right)\norm{u}- C\norm{u}^{q-1}\right)>0,\]
 whenever $\norm{u}=\rho.$
 This completes the proof.
 \end{proof}

\begin{lemma}\label{geo2}
 There exists $e\in D_0^{2,2} (\Omega)$ with $\norm{e} > \rho $ such that
 \[J(e) < \int_{\norm{u}=\rho} J(u).\]
\end{lemma}
\begin{proof}
 Let $0\neq u \in D_0^{2,2} (\Omega)$ and $u\geq 0.$ By (H2) and (H3), there exist $c>0$ and $d>0$ such that
 \be
 F(\xi,s) \geq cs^\theta -d,\,\, \forall \, (\xi,s) \in \Omega \times \R^+, \text{ where } \theta >2.
 \ee
 For $t>0,$ we have
 \be\label{mpg11}
 J(tu) \leq \frac{t^2}{2} \int_\Omega |\Delta_{\mathbb{H}} u|^2 d\xi -ct^\theta \int_\Om \frac{u^\theta}{\rho(\xi)^a} dx +d\int_\Omega \frac{1}{\rho(\xi)^a} d\xi.
 \ee
 Since $\theta >2,$ \eqref{mpg11} implies that $J(tu) \rar -\infty $ as $t\rar \infty.$ By setting $e=tu$ with $t$ large enough, we get $\norm{e}>\rho$ and
 \[J(e) < \inf_{\norm{u}=\rho} J(u).\]
 This completes the proof.
 \end{proof}

 \begin{lemma}\label{psc}
 Assume that $f$ satisfies subcritical growth condition \eqref{subcrit}. Then the functional $J$ satisfies Palais-Smale condition at level $c,$ for all $c\in\R.$
 \end{lemma}
 \begin{proof}
 Let $\{u_k \}\subseteq D_0^{2,2} (\Om)$ ba a PS sequence at level c, that is,
 \be\label{ps1}
 J(u_k)=\frac{1}{2}\norm{u_k}^2-\int_\Omega\frac{F(\xi,u_k)}{\rho(\xi)^a} d\xi \rar c,\,\, \text{as }k\rar \infty
 \ee
 and
 \be\label{ps2}
 |DJ(u_k) v| =\left|\int_\Om \Delta_{\mathbb{H}} u_k\Delta_{\mathbb{H}} v d\xi-\int_\Om \frac{f(\xi,u_k)v}{\rho(\xi)^a}\right| \leq \epsilon_k \norm{v},
 \ee
 where $\epsilon_k\rar 0$ as $k\rar \infty.$
 On taking $v=u_k$ in \eqref{ps2}, we get
 \be\label{ps3}
 |DJ(u_k) u_k| =\left|\int_\Om |\Delta_{\mathbb{H}} u_k|^2d\xi-\int_\Om \frac{f(\xi,u_k)u_k}{\rho(\xi)^a}\right| \leq \epsilon_k \norm{u_k},
 \ee
 On multiplying \eqref{ps1} with $\theta$ and subtracting \eqref{ps3} from it, we get
 \be\label{ps4}
 \left(\frac{\theta}{2}-1\right)\norm{u_k}^2+\int_\Omega \frac{1}{\rho(\xi)^a}(f(\xi,u_k)u_k-\theta F(\xi,u_k)) dx \leq O(1)+\epsilon_k \norm{u_k}.
 \ee
 By (H6), there exist $R_0>0$ and $\theta>2$ such that, for $\norm{u}\geq R_0,$
 \be\label{ps5}
 \theta F(\xi,u) \leq uf(\xi,u).
 \ee

\noindent On using \eqref{ps5}, in \eqref{ps4}, we get
\be\label{ps6}
\left(\frac{\theta}{2}-1\right)\norm{u_k}^2 \leq O(1) +\epsilon_k \norm{u_k}.
\ee
Since $\theta>2,$ \eqref{ps6} shows that $\{u_k\}$ is bounded, therefore, up to a subsequence
\Bea
u_k &\rightharpoonup & u_0\,\,\,\text{in } D_0^{2,2} (\Om)\\
u_k &\rar & u_0 \,\,\,\text{in } L^p(\Om), \, \forall p\geq 1\\
u_k(\xi) &\rar & u_0(\xi)\,\,\text{a.e. in }\Om.
\Eea
Since $f$ has subcritical growth on $\Om,$ therefore there exists a constant $C_k>0$ such that
\be\label{ps7}
f(\xi,s) \leq C_k \exp\left(\frac{A}{2k^2}|s|^{2}\right),\,\,\,\forall\, (\xi,s)\in \Om\times \R.
\ee

Thus
\begin{align}\no
&\left|\int_\Om \frac{f(\xi,u_k)}{\rho(\xi)^a}(u_k-u) d\xi\right|\leq \int_\Om \frac{|f(\xi,u_k)|}{\rho(\xi)^a}|(u_k-u)| dx \\ \no
& \leq \int_\Om C_k \frac{\exp\left(\frac{A}{2k^2}|u_k|^{2}\right)}{\rho(\xi)^a}|u_k-u| d\xi\\ \no
& \leq C\left(\int_\Om \frac{\exp\left(\frac{rA  \norm{u_k}^{2}}{k^2}\frac{|u_k|^{2}}{\norm{u_k}^{2}}\right)}{\rho(\xi)^{ar}}\right)^{\frac{1}{r}}\left(\int_\Om |u_k-u|^{r^\prime}\right)^{\frac{1}{r^\prime}},\\ \no
&\,\,(\text{where $r>1$ and such that $ar<4$ and  } \frac{1}{r}+\frac{1}{r^\prime}=1)\\ \no
& \leq C \norm{u_k-u}_{r^\prime}\,\,\,(\text{by Theorem \ref{sing_adams} })\\
& \rar 0 \text{ as } k\rar \infty.
\end{align}
Similarly, we can show that
\be
\int_\Om \frac{f(\xi,u)}{\rho(\xi)^a}(u_k-u)d\xi\rar 0  \,\,\,\text{as } {k\rar \infty}.
\ee
Also, we have
\[\langle DJ(u_k)-DJ(u), u_k-u \rangle \rar 0, \text{ as } k\rar \infty. \]
Thus $u_k\rar u$ in $D_0^{2,2} (\Om)$. This completes the proof.
\end{proof}
\subsection{Subcritical growth. Proof of Theorem\,\ref{mainthm_sub}}

Using lemmas \ref{geo1}, \ref{geo2} one can show that $J$ satisfies the geometric requirements of mountain pass theorem. Also Lemma \ref{psc}, shows that  $J$ satisfies Palais-Smale conditions. Therefore, we
 conclude the proof of Theorem \ref{mainthm_sub} by applying  mountain pass theorem.
\subsection{The critical growth }
 In this case, we need the following lemma to establish the existence of solution.
\begin{lemma}\label{maxlem}Assume that $f$ satisfies critical exponential growth condition \eqref{crit1} and \eqref{crit2} and suppose (H1)-(H5) hold.
Then there exists $k>0$ such that
\[\max\{J(tA_k): t\geq 0\}< \left(\frac{4-a}{8}\right)\frac{A }{\al_0},\]
where $A_k=\tilde{A}_{R/k}$ is defined by \eqref{adam_func}.
\end{lemma}
\begin{proof}
 We shall prove this result by method of contradiction. Suppose that for all $k,$ we have
\be\label{max1}
\max\{J(tA_k): t\geq 0\}\geq \left(\frac{4-a}{8}\right)\frac{A }{\al_0}.
 \ee
 Therefore for all $k$ there exists a $t_k>0$ at which maximum is attained and
 \be\label{max2}
 J(t_kA_k)=\frac{t_k^2\norm{A_k}^2}{2} -\int_\Om \frac{F(\xi,t_kA_k)}{\rho(\xi)^a} dx \geq \left(\frac{4-a}{8}\right)\frac{A }{\al_0}
 \ee
 and
 \be\label{max3}
 t_k^2\norm{A_k}^2 =\int_\Om  \frac{t_kA_k f(\xi,t_kA_k)}{\rho(\xi)^a} d\xi.
\ee
 Since $F(\xi,s) \geq 0$ and $\norm{A_k}^2 \leq 1,$ therefore from \eqref{max2}, we  get
 \be\label{max4}
 t_k^2 \geq \left(\frac{4-a}{4}\right) \frac{A }{\al_0}.
 \ee
 Also for a given $\tau>0,$ there exists $R_\tau>0$ such that for all  $u\geq R_\tau,$  we have
 \be\label{max5}
 uf(\xi,u)\geq (\beta_1-\tau)\exp(\al_0|u|^{2}).
 \ee
 On using \eqref{max5} in \eqref{max3}, we get
 \begin{align}\no
  t_k^2 &\geq (\beta_1-\tau)\int_{B_{R/k}}\frac{\exp (\al_0 |t_kA_k|^{2})}{\rho(\xi)^a} dx\\ \no
  & =(\beta_1-\tau) \frac{w_{3}}{4-a}\left(\frac{R}{k}\right)^{4-a} \exp\left(\al_0t_k^{2}\left(\frac{4\log k}{A }\right)^{2}\right)\\ \no
  &= (\beta_1-\tau)\frac{w_{3} R^{4-a}}{4-a}\exp\left[\al_0t_k^{2}\left(\frac{4\log k}{A }\right)^{2}-(4-a)\log(k)\right]\\ \label{max6}
  1&\geq (\beta_1-\tau)\frac{w_3 R^{4-a}}{4-a}\exp\left[\al_0t_k^{2}\left(\frac{4\log k}{A }\right)^{2}-(4-a)\log(k)-2\log(t_k)\right].
\end{align}
\eqref{max6} shows that $\{t_k\}$ is a bounded sequence, otherwise up to a subsequence right hand side of \eqref{max6} tends to $\infty$ as $k\rar \infty.$ Also, we have
\be\label{max7}
t_k^2 \rar \left(\frac{4-a}{4}\right) \frac{A }{\al_0}\,\,\text{ as } k\rar\infty
\ee
and
\be\no
\norm{A_k} \rar 1\,\,\text{ as } k\rar\infty.
\ee
Also observe that, by definition of $A_k,$ as $k\rar \infty$, we have,
\be\no
A_k(\xi) \rar 0,\,\,\, \text{a.e. } \xi\in \Omega.
\ee
Let
\[X_k=\{\xi\in \Om : t_kA_k \geq R_\tau\}\]
and
\[Y_k=\Om\backslash X_k,\]
 then the characteristic function of $Y_k,$ $\chi_{Y_k}\rar 1, $ a.e. $\xi\in\Om.$ By Lebesgue dominated convergence theorem, we get
 \be\label{max8}
 \int_{Y_k} t_kA_k \frac{f(\xi,t_kA_k)}{\rho(\xi)^a} d\xi \rar 0
 \ee
 and
 \be\label{max9}
 \int_{Y_k} \frac{\exp (\al_0 |t_kA_k|^2)}{\rho(\xi)^a} dx \rar \frac{w_3R^{4-a}}{4-a}, \text{ as } k\rar \infty.
 \ee
 Since $\displaystyle t_k^2  \geq \frac{4-a}{4} \frac{A }{\alpha_0},$ therefore
 \begin{align}\no
  &\int_{B_R} \frac{\exp (\al_0 |t_kA_k|^2)}{\rho(\xi)^a} d\xi \geq \int_{B_R} \frac{\exp \left(\frac{4-a}{4}A  |A_k|^2\right)}{\rho(\xi)^a } d\xi \\ \no
  & = \int_{\|\xi\|\leq \frac{R}{k}} \frac{\exp \left(\frac{4-a}{4}A  |A_k|^2\right)}{\rho(\xi)^a } d\xi+\int_{\frac{R}{k}\leq \|\xi\| \leq R} \frac{\exp \left(\frac{4-a}{4}A  |A_k|^2\right)}{\rho(\xi)^a } d\xi\\ \no
  & = \int_{\|\xi\|\leq \frac{R}{k}} \frac{\exp \frac{4-a}{4}(A  |A_k|^2)}{\rho(\xi)^a } d\xi+ \int_{\frac{R}{k}\leq \|\xi\| \leq R} \frac{\exp \left(\frac{4-a}{4}A  |A_k|^2\right)}{\rho(\xi)^a } d\xi\\
  &=\frac{w_3R^{4-a}}{4-a} +R^{4-a} \mathcal{M}
 \end{align}
 Since
 \begin{gather*}
 t_k^2 \geq (\beta_1-\tau) \int_{\|\xi\|\leq R} \frac{\exp (\al_0 |t_kA_k|^2)}{\rho(\xi)^a } d\xi +\int_{Y_k} \frac{t_kA_k f(\xi,t_kA_k)}{\rho(\xi)^a } d\xi \\
 -(\beta_1-\tau)\int_{Y_k} \frac{\exp (\al_0|t_kA_k|^2)}{\rho(\xi)^a } d\xi ,
\end{gather*}
therefore
\[\frac{4-a}{4}\frac{A }{\al_0} \geq (\beta_1-\tau) R^{4-a}\mathcal{M}\]
or
\[\beta_1\leq \frac{A }{R^{4-a}\mathcal{M}\al_0}\frac{4-a}{4},\]
which is a contradiction to (H5).
This completes the proof.
\end{proof}

\begin{lemma}\label{psc1}
 Assume that $f$ satisfies critical exponential growth condition \eqref{crit1} and \eqref{crit2}. Let $\{u_k\} \subseteq D_0^{2,2} (\Om)$ be a Palais-Smale sequence. Then $\{u_k\}$ has a subsequence, still denoted by $\{u_k\},$ and $u\in D_0^{2,2} (\Om)$ such that
 \begin{enumerate}[(i)]
  \item $u_k \rightharpoonup u$ in $D_0^{2,2} (\Om)$

  \item $\displaystyle \frac{f(\xi,u_k)}{\rho(\xi)^a } \rar  \frac{f(\xi,u)}{\rho(\xi)^a }$ in $L^1(\Om).$
 \end{enumerate}
\end{lemma}
\begin{proof}
 Let $\{u_k\}$ be a Palais-Smale sequence, then
 \be\label{ps11}
 J(u_k)=\frac{1}{2} \norm{u_k}^2-\int_\Om F(\xi,u_k) d\xi \rar c,\,\ \text{as } k\rar \infty
 \ee
 and
 \be\label{ps22}
 |J^\prime (u_k)v| =\left|\int_\Om \Delta_{\mathbb{H}} u_k \Delta_{\mathbb{H}} v d\xi -\int_\Om f(\xi,u_k) v d\xi \right|\leq \tau_k \norm{v}.
 \ee
 Also by Lemma \ref{maxlem},
 \[c< \frac{4-a}{8}\frac{A }{\alpha_0}.\]
 From \eqref{ps11} and \eqref{ps22}, we get
 \begin{align}\no
 C+\tau_n\norm{u_k}&\geq \left(\frac{\theta}{2}-1\right)\norm{u_k}^2-\int_\Om \frac{(\theta F(\xi,u_k)-f(\xi,u_k)u_k)}{\rho(\xi)^a }d\xi \\ \label{ps33}
 & \geq \left(\frac{\theta}{2}-1\right)\norm{u_k}^2,
 \end{align}
 which implies that
 \be\label{ps44}
 \begin{cases}
 \displaystyle\norm{u_k} \leq  C,\\
  \displaystyle\int_\Om \frac{f(\xi,u_k)u_k}{\rho(\xi)^a }d\xi  \leq  C,\\
  \displaystyle \int_\Om \frac{F(\xi,u_k)}{\rho(\xi)^a } d\xi  \leq  C.
 \end{cases}
 \ee
Since $D_0^{2,2} (\Om)$ is a reflexive Banach space, therefore by \eqref{ps44}, up to a subsequence
\[\begin{cases}
 \displaystyle u_k \rightharpoonup u \text{ in } D_0^{2,2} (\Om),\\
 u_k\longrightarrow u \text{ in } L^q(\Om),\, \forall\,  1\leq q<\infty,\\
 u_k(\xi)\longrightarrow u(\xi), \text{ a.e.  } \xi\in \Om.
\end{cases} \]
Furthermore, using the  arguments similar to  Lemma 2.1 \cite{fig}, we get
\be
\frac{f(\xi,u_n)}{\rho(\xi)^a } \rar \frac{f(\xi,u)}{\rho(\xi)^a } \text{ in }  L^1(\Om).
\ee
This completes the proof.
\end{proof}
\subsection{Proof of Theorem\,\ref{mainthm_crit}}
By Lemmas \ref{geo1}, \ref{geo2}, we can find  a Palais-Smale sequence $\{u_k\}$ at the level $c$ and by Lemma \ref{maxlem},
$\displaystyle 0<c<\frac{4-a}{8}\frac{A }{\al_0}.$ Thus, we have
\be\label{f0}
J(u_k)=\frac{1}{2}\norm{u_k}^2-\int_\Omega F(x,u_k)d\xi \longrightarrow c
\ee
and
\be\lab{f1}
|J^\prime(u_k)v|=\left|\int_\Omega \Delta_{\mathbb{H}} u_k\Delta_{\mathbb{H}} vdx-\int_\Omega \frac{f(\xi,u_k)v}{\rho(\xi)^a }d\xi\right|\leq \epsilon_k\norm{v}.
\ee
By Lemma \ref{psc1}, there exists $u\in D_0^{2,2} (\Omega)$ such that
\begin{enumerate}[(i)]
\item $u_k \rightharpoonup u$ in $D_0^{2,2} (\Omega).$
\item $\frac{f(\xi,u_k)}{\rho(\xi)^a } \rar \frac{f(\xi,u)}{\rho(\xi)^a }$ strongly in $L^1(\Omega)$.
\end{enumerate}
Therefore by \eqref{f1},  with the aid of Lebesgue dominated convergence theorem, one can pass the limit and get
\[J^\prime(u) v=0\] for all $v\in C_c^\infty(\Omega).$ Since $C_c^\infty(\Omega)$ is dense in $D_0^{2,2} (\Omega),$ therefore $u$ is a weak solution to \eqref{main_prob}.

Now, we show that $u$ is non trivial. On the contrary, let if possible $u\equiv 0,$ then by (H2) and Lebesgue dominated convergence theorem,
\be
\int_\Om \frac{F(\xi,u_k)}{\rho(\xi)^a } d\xi \rar 0 \text{ in } L^1(\Omega) \text{ as }k\rar\infty.
\ee
From \eqref{f0}, we get
\be
\norm{u_k}^2 \rar 2c<\frac{4-a}{4}\frac{A }{\al_0}.
\ee
Choose $q>1,$ sufficiently close to $1$ such that
\[\frac{4}{4-a}q\alpha_0\norm{u_k}^{2} <A \]
for $k$ large.
Now, since $f$ has critical exponential growth, therefore by Theorem \ref{sing_adams},
\begin{align*}
\int_\Om \frac{|f(\xi,u_k)}{\rho(\xi)^a } d\xi &\leq C\int_\Om \exp\left(q\alpha_0\norm{u_k}^{2} \left|\frac{u_k}{\norm{u_k}}\right|^{2}\right) d\xi \\
&\leq O(1),\,\,\,\text{ as } k\rar \infty.
\end{align*}
Thus, by taking $v=u_k$ in \eqref{f0}, we obtain
\[\norm{u_k}^2 \rar 0\text{ as } k\rar \infty,\]
which is a contradiction. This completes the proof.
\qed
\subsection{The critical potential case $a=4$}
In this section, we consider the borderline problem with respect to potential, i.e.,  $a=4$
\begin{equation}\label{critical_prob}
\left.
  \begin{array}{ll}
    \Delta_{\mathbb{H}}^2u=\displaystyle\frac{f(\xi,u)}{\rho(\xi)^4}  \,\,\,\, \mbox{in } \Omega, \\ \\
     u|_{\pa\Om}=0=\displaystyle\left.\frac{\partial u}{\partial n}\right|_{\pa\Om},
  \end{array}
\right.
\end{equation}
where $0\in\Om\subseteq \H,$ is a bounded domain and $f$ satisfies the exponential growth condition at subcritical and critical level.
This case is delicate in the sense that Theorems \ref{mainthm_sub} and \ref{mainthm_crit} fail when $a=4.$

In order to establish the existence of solution to the problem \eqref{critical_prob}, we consider the approximate problem which has subcritical potential
\begin{equation}\label{approx_prob}
\left.
  \begin{array}{ll}
    \Delta_{\mathbb{H}}^2u_n=\displaystyle\frac{f(\xi,u_n)}{\rho(\xi)^{4-\frac{1}{n}}} \,\,\,\, \mbox{in } \Omega, \\
         u_n|_{\pa\Om}=0=\displaystyle\left.\frac{\partial u_n}{\partial n}\right|_{\pa\Om},
  \end{array}
\right.
\end{equation}
The solutions to \eqref{approx_prob} are the critical points of the functional
\[J_n: D_0^{2,2} (\Omega) \rar \R\] defined as
\be\label{funca}
J_n(u_n)=\frac{1}{2} \int_\Om |\Delta_{\mathbb{H}} u_n|^2 d\xi -\int_\Om \frac{F(\xi,u_n)}{\rho(\xi)^{4-\frac{1}{n}}} d\xi ,
\ee
where $\displaystyle F(\xi,u_n)=\int_0^{u_n} f(\xi,s) ds.$
\begin{lemma}\lab{geoa1}
  Suppose (H1)-(H4) hold. Then there exists $\rho>0 $ such that
 \[J_n(u_n)>0,\, \, \text{if } \norm{u_n}=\rho.\]
\end{lemma}
\begin{proof}
The proof has the similar lines as the proof of Lemma \ref{geo1}, for the sake of brevity, we omit the details.
\end{proof}
\begin{lemma}\label{geoa2}
 There exists $e_n\in D_0^{2,2} (\Omega)$ with $\norm{e_n} > \rho $ such that
 \[J_n(e_n) < \int_{\norm{u_n}=\rho} J_n(u_n).\]
\end{lemma}
\begin{proof}
 The proof has similar lines as the proof of Lemma \ref{geo2} and therefore we omit the details for the sake of brevity.
\end{proof}

\begin{lemma}\label{psca}
 The functional $J_n$ satisfies Palais-Smale condition at level $c,$ for all $c\in\R.$
 \end{lemma}
 \begin{proof}
 Let $\{u_n^{(m)} \}\subseteq D_0^{2,2} (\Om)$ ba a (PS) sequence at level c, that is,
 \be\label{psa1}
 J_n(u_n^{(m)})=\frac{1}{2}\norm{u_n^{(m)}}^2-\frac{F(\xi,u_n^{(m)})}{\rho(\xi)^{4-\frac{1}{n}}} d\xi \rar c,\,\, \text{as }m\rar \infty
 \ee
 and
 \be\label{psa2}
 |DJ_n(u_n^{(m)}) v| =\left|\int_\Om \Delta_{\mathbb{H}} u_n^{(m)}\Delta_{\mathbb{H}} v d\xi-\int_\Om \frac{f(\xi,u_n^{(m)})v}{\rho(\xi)^{4-\frac{1}{n}}}\right| \leq \epsilon_m \norm{v},
 \ee
 where $0<\epsilon_m<1$ and $\epsilon_m\rar 0$ as $m\rar \infty.$
 On taking $v=u_n^{(m)}$ in \eqref{psa2}, we get
 \be\label{psa3}
 |DJ_n(u_n^{(m)}) u_n^{(m)}| =\left|\int_\Om |\Delta_{\mathbb{H}} u_n^{(m)}|^2d\xi-\int_\Om \frac{f(\xi,u_n^{(m)})u_n^{(m)}}{\rho(\xi)^{4-\frac{1}{n}}}\right| \leq \epsilon_m \norm{u_n^{(m)}},
 \ee
 On multiplying \eqref{psa1} with $\theta$ and subtracting \eqref{psa3} from it, we get
 \be\label{psa4}
 \left(\frac{\theta}{2}-1\right)\norm{u_n^{(m)}}^2+\int_\Omega \frac{1}{\rho(\xi)^{4-\frac{1}{n}}}(f(\xi,u_n^{(m)})u_n^{(m)}-\theta F(\xi,u_n^{(m)})) d\xi \leq O(1)+\epsilon_m \norm{u_n^{(m)}}.
 \ee
 By (H6), there exist $R_0>0$ and $\theta>2$ such that, for $\norm{u_n}\geq R_0,$
 \be\label{psa5}
 \theta F(\xi,u_n) \leq u_nf(\xi,u_n).
 \ee
 On using \eqref{psa5}, in \eqref{psa4}, we get
\be\label{psa6}
\left(\frac{\theta}{2}-1\right)\norm{u_n^{(m)}}^2 \leq O(1) +\epsilon_m \norm{u_n^{(m)}}.
\ee
Since $\theta>2,$ \eqref{psa6} shows that $\{u_n^{(m)}\}$ is bounded for each fixed $n\in\mathbb{N}$, that is, $\norm{u_n^{(m)}}\leq K_n,$ for some $K_n>0$ and therefore, up to a subsequence, we have
\Bea
u_n^{(m)} &\rightharpoonup & w_n\,\,\,\text{in } D_0^{2,2} (\Om) \text{ as } m\rar\infty.\\
u_n^{(m)} &\longrightarrow & w_n \,\,\,\text{in } L^p(\Om), \text{ as } m\rar\infty \, \text{ for all } p\geq 1. \\
u_n^{(m)}(\xi) &\longrightarrow & w_n(\xi)\,\,\text{a.e. in }\Om,  \text{ as }\,m\rar\infty.
\Eea
Since $f$ has subcritical growth on $\Om,$ therefore there exists a constant $C_{K_n}>0$ such that
\be\label{psa7}
f(\xi,s) \leq C_{K_n} \exp\left(\frac{\beta_n}{2K_n^2}|s|^{2}\right),\,\,\,\forall\, (\xi,s)\in \Om\times \R,
\ee
where $\beta_n=A\left(4-a+\frac{1}{n}\right).$
Thus
\begin{align}\no
&\left|\int_\Om \frac{f(\xi,u_n^{(m)})}{\rho(\xi)^{4-\frac{1}{n}}}(u_n^{(m)}-w_n) d\xi\right|\leq \int_\Om \frac{|f(\xi,u_n^{(m)})|}{\rho(\xi)^{4-\frac{1}{n}}}|(u_n^{(m)}-w_n)| d\xi \\ \no
& \leq \int_\Om C_{K_n} \frac{\exp\left(\frac{\beta_n}{2K_n^2}|u_n^{(m)}|^2\right)}{\rho(\xi)^{4-\frac{1}{n}}}|u_n^{(m)}-w_n| d\xi\\ \no
& \leq C\left(\int_\Om \frac{\exp\left(\frac{r\beta_n \norm{u_n^{(m)}}^{2}}{K_n^2}\frac{|u_n^{(m)}|^{2}}{\norm{u_n^{(m)}}^{2}}\right)}{\rho(\xi)^{{(4-\frac{1}{n})}r}}\right)^{\frac{1}{r}}\left(\int_\Om |u_n^{(m)}-w_n|^{r^\prime}\right)^{\frac{1}{r^\prime}}\\ \no
&\,\,\left(\text{where $r>1$ and such that $\left(4-\frac{1}{n}\right)r>4$  }\right.\\
&\left. \text{ and }\frac{1}{r}+\frac{1}{r^\prime}=1\right) \nonumber \\
& \leq C \norm{u_n^{(m)}-w_n}_{r^\prime}\nonumber \\
& \rar 0 \text{ as } m\rar \infty.
\end{align}
Similarly, we can show that
\be
\int_\Om \frac{f(\xi,u_n^{(m)})}{\rho(\xi)^{4-\frac{1}{n}}}(u_n^{(m)}-w_n)d\xi\rar 0  \,\,\,\text{as } {m\rar \infty}.
\ee
Also, we have
\[\langle DJ(u_n^{(m)})-DJ(w_n), u_n^{(m)}-w_n \rangle \rar 0, \text{ as } m\rar \infty. \]
Thus $u_n^{(m)}\rar w_n$ in $D_0^{2,2} (\Om).$ This completes the proof.
\end{proof}
\subsection{Proof of Theorem\,\ref{mainthmapprox_sub}}

Lemmas \ref{geoa1}, \ref{geoa2} show that the functional $J_n$ satisfies the geometric conditions required in mountain pass theorem. Lemma \ref{psca} shows that $J_n$
satisfies Palais-Smale condition and therefore
by mountain pass theorem, we conclude that Problem \eqref{approx_prob} has a weak solution $u_n,$ for each $n,$ that is,
\be\lab{weak_soln}
\int_\Omega \Delta_{\mathbb{H}} u_n\Delta_{\mathbb{H}} v d\xi =\int_\Om \frac{f(\xi,u_n)}{\rho(\xi)^{4-\frac{1}{n}}}vdx,\,\,\,\text{for all $v\in D_0^{2,2} (\Omega).$}
\ee
Since $0<\epsilon_m<1$ therefore from Equation \eqref{psa6}, we have $\norm{u_n}\leq C,$ for some constant $C$ independent of $n.$ Since $D_0^{2,2}(\Om)$ is reflexive Banach space therefore, up to a subsequence
\Bea
u_n &\rar & u_0\,\,\,\text{in } D_0^{2,2} (\Om)\\
u_n &\rar & u_0 \,\,\,\text{in } L^p(\Om), \, \forall p\geq 1\\
u_n(\xi) &\rar & u_0(\xi)\,\,\text{a.e. in }\Om.
\Eea
From \eqref{psa6} and the arguments used in Lemma \ref{psc1}, we also have the following
\be
\int_\Omega \frac{f(\xi,u_n)u_n}{\rho(\xi)^{4-\frac{1}{n}}} d\xi \leq C
\ee
and
\be
\int_\Omega \frac{F(\xi,u_n)}{\rho(\xi)^{4-\frac{1}{n}}} d\xi \leq C.
\ee
Observe that
\be\lab{conv}
\frac{f(\xi,u_n)}{\rho(\xi)^{4-\frac{1}{n}}} \rar \frac{f(\xi,u_0)}{\rho(\xi)^4}, \text{ a.e. in }\Omega .
\ee
Using \eqref{conv} and Vitali's convergence theorem in \eqref{weak_soln}, we get that $u_0$ is a weak solution of \eqref{critical_prob}.
This completes the proof in the subcritical case.
\qed

Now, we establish the existence of solution to \eqref{critical_prob}, when $f$ satisfies critical exponential growth condition \eqref{crit1} and \eqref{crit2}.

\subsection{Proof of Theorem\,\ref{mainthmapprox_crit}}
Since for each $n\in \mathbb{N},$ $4-\frac{1}{n}<4,$ therefore by Theorem \ref{mainthm_crit}, \eqref{approx_prob} has a weak solution $u_n.$ Moreover, since  $0<\epsilon_m<1$ therefore by \eqref{ps44},
there exists $C>0$ independent of $n$ such that
$\norm{u_n}\leq C,$
therefore, up to a subsequence
\Bea
u_n &\rar & u_0\,\,\,\text{in } D_0^{2,2}  (\Om).\\
u_n &\rar & u_0 \,\,\,\text{in } L^p(\Om), \, \forall p\geq 1.\\
u_n(\xi) &\rar & u_0(\xi)\,\,\text{a.e. in }\Om.
\Eea
From \eqref{psa6} and the arguments used in Lemma \ref{psc1}, we also have the following
\be
\int_\Omega \frac{f(\xi,u_n)u_n}{\rho(\xi)^{4-\frac{1}{n}}} d\xi \leq C
\ee
and
\be
\int_\Omega \frac{F(\xi,u_n)}{\rho(\xi)^{4-\frac{1}{n}}} d\xi \leq C.
\ee
Observe that
\be\lab{conv1}
\frac{f(\xi,u_n)}{\rho(\xi)^{4-\frac{1}{n}}} \rar \frac{f(\xi,u_0)}{\rho(\xi)^4}, \text{ a.e. in }\Omega .
\ee
Using \eqref{conv1} and Vitali's convergence theorem in \eqref{weak_soln}, we get that $u_0$ is a weak solution of \eqref{critical_prob}.
This completes the proof in the critical case. \qed
\section*{Acknowledgement}
Authors thank the referee for constructive comments.

\end{document}